\pdfoutput=1
\documentclass[a4paper]{amsart}
\usepackage{fixltx2e}
\usepackage{mathtools}
\usepackage{booktabs}
\usepackage[utf8]{inputenc}
\usepackage[T1]{fontenc}
\usepackage[charter]{mathdesign}
\usepackage[bookmarks=true,pdfdisplaydoctitle=true]{hyperref}
\hypersetup{
pdfpagemode=UseNone,
pdfstartview=FitH,
pdftitle={Corners over quasirandom groups},
pdfauthor={Pavel Zorin-Kranich},
pdflang=en-US
}

\usepackage[safeinputenc=true,style=alphabetic,firstinits,maxnames=4,doi=false,eprint=true,isbn=false,url=false,backend=bibtex]{biblatex}

\newbibmacro{string+doiurl}[1]{%
  \iffieldundef{doi}{%
    \iffieldundef{url}{%
         #1%
      }{%
      \href{\thefield{url}}{#1}%
    }%
  }{%
    \href{http://dx.doi.org/\thefield{doi}}{#1}%
  }%
}
\DeclareFieldFormat[article,misc,incollection]{title}{\usebibmacro{string+doiurl}{\mkbibquote{#1}}}

\newcommand*{\mailto}[1]{\href{mailto:#1}{#1}}

\numberwithin{equation}{section}
\newtheorem{theorem}[equation]{Theorem}
\newtheorem{lemma}[equation]{Lemma}
\newtheorem{proposition}[equation]{Proposition}

\theoremstyle{definition}

\theoremstyle{remark}

\newcommand*{\N}{\mathbb{N}}
\newcommand*{\Z}{\mathbb{Z}}
\newcommand*{\R}{\mathbb{R}}
\newcommand*{\dif}{\mathrm{d}}

\DeclarePairedDelimiter\abs{\lvert}{\rvert}
\DeclarePairedDelimiter\card{\lvert}{\rvert}
\DeclarePairedDelimiter\norm{\lVert}{\rVert}
\DeclarePairedDelimiter\Set\{\}
\DeclarePairedDelimiterX\innerp[2]{\langle}{\rangle}{#1,#2}
\allowdisplaybreaks

\begin{document}
\title{Corners over quasirandom groups}
\author{Pavel Zorin-Kranich}
\address{Universit\"at Bonn\\
  Mathematisches Institut\\
  Endenicher Allee 60\\
  53115 Bonn\\
  Germany
}
\email{\mailto{pzorin@math.uni-bonn.de}}
\urladdr{\url{http://www.math.uni-bonn.de/people/pzorin/}}
\begin{abstract}
Let $G$ be a finite $D$-quasirandom group and $A \subset G^{k}$ a $\delta$-dense subset.
Then the density of the set of side lengths $g$ of corners
\[
\Set{(a_{1},\dotsc,a_{k}),(ga_{1},a_{2},\dotsc,a_{k}),\dotsc,(ga_{1},\dotsc,ga_{k})} \subset A
\]
converges to $1$ as $D\to\infty$.
\end{abstract}
\maketitle

\section{Notation and background}
In this article we will be concerned with a version of the multidimensional Szemer\'edi theorem over quasirandom groups.
In order to state our results and put them into historical perspective, we begin by introducing appropriate notation.
Let $G$ be a countable group and let $T_{i}$ be the commuting $G$-actions on $G^{k}$ given by
\[
T_{i}^{g}(a_{1},\dotsc,a_{k}) := (a_{1},\dotsc,a_{i-1},ga_{i},a_{i+1},\dotsc,a_{k}).
\]
We write $T_{[j,i]}^{g} := T_{j}^{g}\dotsm T_{i}^{g}$.
A \emph{(BMZ) corner} in $G^{k}$ is a configuration of the form
\begin{equation}
\label{eq:corner}
C(g,\vec a) := \Set{\vec a, T_{[1,1]}^{g}\vec a, \dotsc, T_{[1,k]}^{g}\vec a},
\qquad
g \in G,
\vec a\in G^{k}.
\end{equation}
We call $\vec a$ the \emph{base point} and $g$ the \emph{side length} of a corner.
A corner is called \emph{non-trivial} if its side length is distinct from $1_{G}$.

BMZ corners are not the only natural configurations generalizing the corners that appear in the commutative situation $G=\Z$.
However, they seem to be the best behaved ones.
Resolving a conjecture from \cite{MR1481813}, Austin has recently proved that if $G$ is amenable and $A\subset G^{k}$ has positive upper Banach density, then $A$ contains (many) non-trivial BMZ corners.
This extends several previous results.
The case $G=\Z$ is the mulidimensional Szemer\'edi theorem due to Furstenberg and Katznelson, from which the original Szemer\'edi theorem on arithmetic progressions in $\Z$ can be deduced using the projection map $\Z^{k}\to\Z$, $\vec a \mapsto a_{1}+\dotsb+a_{k}$.
The cases $k=2$ of all these results have been known prior to the general cases as indicated in the table below.
\begin{center}
\begin{tabular}{lccc}
\toprule
 & $\Z$ & $\Z^{k}$ & $G^{k}$ \\
\midrule
$k=2$ & \cite{MR0051853} & \cite{MR0369299} & \cite{MR1481813} \\
$k\geq 3$ & \cite{MR0369312} & \cite{MR531279} & \cite{arXiv:1309.4315} \\
\bottomrule
\end{tabular}
\end{center}
A finitary version of the multidimensional Szemer\'edi theorem reads as follows.
\begin{theorem}[{\cite[Theorem 11]{MR3177376}}]
\label{thm:corners}
Let $\delta>0$ and $k\in\N$.
Then there exist $\epsilon>0$ and $N\in\N$ such that for every finite group $G$ with $\card{G}>N$ every subset $A\subset G^{k}$ with $\card{A}>\delta \card{G}^{k}$ contains at least $\epsilon \card{G}^{k+1}$ BMZ corners.
\end{theorem}
This theorem is an easy consequence of Gowers's hypergraph removal lemma \cite{MR2373376}, and we reproduce the proof here in order to motivate both the definition of the BMZ corners and the change of variables that will be used in the proof of Theorem~\ref{thm:multiple-weak-mixing}.
\begin{proof}[Proof of Theorem~\ref{thm:corners}]
Here and later we denote omission of the $i$-th coordinate in a vector by a subscript as follows:
\[
\vec x_{(i)} = (x_{0},\dotsc,x_{i-1},x_{i+1},\dotsc,x_{k})
\]
For $i=0,\dotsc,k$ consider the changes of variables
\[
N^{[0,k]}_{i}(\vec x_{(i)})
:=
(x_{0},x_{0}x_{1},\dotsc,(x_{0}\cdots x_{i-1}),(x_{i+1}\cdots x_{k})^{-1},\dotsc,x_{k}^{-1}).
\]
They are related among each other as follows: if $x\in G^{k+1}$ and $g=x_{0}\dotsm x_{k}$, then $T_{[1,i]}^{g} N^{[0,k]}_{0} (x_{(0)}) = N^{[0,k]}_{i} (x_{(i)})$.
In particular, corners are precisely the configurations
\[
\Set{ N^{[0,k]}_{i} (x_{(i)}), i=0,\dotsc,k},
\quad
x\in G^{k+1}.
\]
Define a $(k+1)$-partite $k$-uniform hypergraph $F$ with vertex sets $X_{0}=\dotsb=X_{k}=G$ by
\[
\vec x_{(i)} \in F
:\iff
N^{[0,k]}_{i}(\vec x_{(i)}) \in A.
\]
Then a corner corresponds to a simplex in the hypergraph $F$.
If there were fewer than $\epsilon \card{G}^{k+1}$ simplices in $F$, then by the hypergraph removal lemma \cite[Theorem 10.1]{MR2373376} the hypergraph $F$ could be made simplex-free by removing fewer than $\delta \card{G}^{k}$ edges.
But if we remove the element of $A$ corresponding to each removed edge and repeat the construction of $F$, we obtain an even smaller hypergraph that still contains simplices (since each remaining member of $A$ gives rise to a trivial corner), a contradiction.
\end{proof}
A similar argument works for $A\subset\Phi^{k}$, where $\Phi\subset G$ is a set with $\card{\Phi^{-1}\Phi} \leq C \card{\Phi}$, with constants depending on $C$.
This proves a version of Theorem~\ref{thm:corners} over infinite amenable groups that admit a F\o{}lner sequence satisfying the Tempelman condition.
This argument does not seem to extend easily to general F\o{}lner sequences.

\section{Main result}
The problem of finding arithmetic progressions, and later more general configurations, in dense subsets of amenable groups has been transferred to ergodic theory by Furstenberg, who reformulated Szemer\'edi's theorem as a multiple recurrence theorem and gave it a new proof \cite{MR0498471}.
An important special case of the multiple recurrence theorem occurs for weakly mixing actions, when its conclusion can be strengthened to the extent that corners with almost every possible side length can be found.

A (necessarily infinite) group is called \emph{weakly mixing} if it has no non-trivial finite-dimensional unitary representations.
For such groups many combinatorial results can be strengthened substantially, see e.g.\ \cite{MR3177376}.
A quantitative notion of weak mixing has been introduced by Gowers \cite{MR2410393}.
A group is called \emph{$D$-quasirandom} if it has no non-trivial unitary representation of dimension less than $D$.
Our result tells that in dense subsets over quasirandom groups one can find corners of almost every side length.
\begin{theorem}
\label{thm:quasirandom-corners}
Let $\delta>0$ and $k\in\N$.
Then for every finite $D$-quasirandom group $G$ and every subset $A\subset G^{k}$ with $\card{A}>\delta \card{G}^{k}$ we have
\[
\card{\Set{ g\in G : \card{\Set{ \vec a \in G^{k} : C(g,\vec a)\subset A}} > \epsilon(\delta,k) \card{G}^{k}/2}} > (1-o_{D\to\infty;\delta,k}(1)) \card{G},
\]
where $\epsilon(\delta,k)$ is the quantity from Theorem~\ref{thm:corners}.
\end{theorem}

The case $k=2$ has been previously shown in \cite{arxiv:1410.5385,arxiv:1503.08746}, and we refer to those articles for further discussion of why BMZ corners are natural.

Since the set $\Set{ \vec a \in G^{k} : C(g,\vec a)\subset A}$ has density at least $\epsilon$ on average (over $g$) by Theorem~\ref{thm:corners}, it suffices to show that its density is usually close to the average.
We formulate this in the language of dynamical systems as a multiple weak mixing property.

\begin{theorem}
\label{thm:multiple-weak-mixing}
Let $G$ be a compact $D$-quasirandom group, $k\geq 0$, and $f_{i} : G^{k}\to [-1,1]$, $i\in [0,k]$.
Consider the multicorrelation sequence
\[
c_{g} := \int_{G^{k}} \prod_{i=0}^{k} f_{i}T_{[1,i]}^{g}.
\]
Then
\[
\int_{g} \abs[\Big]{ c_{g} - \int_{h}c_{h} } = o_{D\to\infty;k}(1).
\]
\end{theorem}
In other words, the multicorrelation sequence converges to its average \emph{in density} as $D\to\infty$.
Here and later, compact groups are equipped with the normalized Haar measure and $fT=f\circ T$ denotes the composition of functions $f$ and $T$.
\begin{proof}[Proof of Theorem~\ref{thm:quasirandom-corners} assuming Theorem~\ref{thm:multiple-weak-mixing}]
Let $f_{0}=\dotsb=f_{k}=1_{A}$.
Then the multicorrelation sequence $c_{g}$ counts the elements $\vec a\in G^{k}$ such that $C(g,\vec a)\subset A$.
On the other hand, by Theorem~\ref{thm:corners} we have $\int_{g}c_{g} > \epsilon(\delta,k)$ provided $\card{G}>D$ is large enough.
\end{proof}

\section{Tools}
In this and the next section $G$ always denotes a compact group with normalized Haar measure.
Quasirandomness will be used in the following form.
\begin{lemma}[{\cite[Corollary 3]{MR3303180}}]
\label{lem:density-mean-ergodic}
Let $V$ be a (real or complex) Hilbert space equipped with an (orthogonal or unitary) right action of a compact $D$-quasirandom group $G$ and let $P$ be the projection onto the invariant subspace.
Then for every $u,v\in V$ we have
\[
\int_{G} \abs{ \innerp{u}{vg} - \innerp{Pu}{Pv} }^{2} \dif g \leq D^{-1} \norm{u}^{2} \norm{v}^{2}.
\]
\end{lemma}
This result has been stated for left actions in \cite{MR3303180}, the version above follows by considering either the adjoint action or the opposite group.

We use the following version of the van der Corput lemma.
\begin{lemma}[{\cite[Lemma 1]{arxiv:1503.08746}}]
\label{lem:vdC}
Let $V$ be a (real or complex) Hilbert space and $u:G\to V$ a measurable function.
Then for every $v\in V$ with $\norm{v}\leq 1$ we have
\[
\int \abs{\innerp{v}{u(g)}} \dif g
\leq
\sqrt{\iint \abs{\innerp{u(g)}{u(h)}} \dif g \dif h}.
\]
\end{lemma}

For a function $F:G^{k}\to\R$, $k\geq 1$, the $k$-variable Gowers box norm is defined by
\[
\norm{F}_{\square^{k}}^{2^{k}}
=
\int \prod_{\vec\epsilon \in \Set{0,1}^{[1,k]}} F(\vec x^{\epsilon}_{(i)})
\dif (x_{j,\epsilon})_{j\in [1,k],\epsilon\in\Set{0,1}},
\]
where $\vec x^{\epsilon}_{j} = x_{j,\epsilon_{j}}$ and $[1,k]=\Set{1,\dotsc,k}$.
See e.g.\ \cite{MR2931680} for a discussion of the basic properties of these norms.

Recall a version of the (weak) weighted hypergraph regularity lemma \cite[Lemma 2.9]{MR2259060}.
This particular version can be found in \cite[Corollary 6.8]{MR2359471} for $k=2$, and the proof for general $k$ is similar.
\begin{lemma}[Weak regularity lemma]
\label{lem:regularity}
For every $k\in\N$ and $\epsilon>0$ there exists $M\in\N$ such that every measurable function $F : G^{k}\to [-1,1]$ can be written
as $F=F_{s}+F_{u}$, where
\begin{enumerate}
\item $F_{s}$ is measurable with respect to $\vee_{j=1}^{k} B_{j}$, where each $B_{j}$ is a $\sigma$-algebra on $G^{k}$ generated by at most $M$ atoms that does not depend on the $j$-th coordinate,
\item $\norm{F_{u}}_{\square^{k}} \leq\epsilon$, and
\item $\abs{F_{s}},\abs{F_{u}}\leq 2$.
\end{enumerate}
\end{lemma}

\section{Multiple weak mixing}
Theorem~\ref{thm:multiple-weak-mixing} is proved by induction on $k$ following these steps:
\begin{enumerate}
\item\label{step:estimate} Prove a Gowers box norm estimate for the average in question.
\item Apply the hypergraph regularity lemma to split one of the functions into a structured and a quasirandom part.
\item Estimate the quasirandom part using step~\ref{step:estimate} and the structured part using the inductive hypothesis.
\end{enumerate}
Step~\ref{step:estimate} in this plan is given by the following estimate.

\begin{proposition}
\label{prop:box-norm-controls-multicorrelations}
Let $G$ be a compact $D$-quasirandom group and $k\geq 1$.
Then for every tuple of functions $f_{i}: G^{k} \to [-1,1]$, $i\in [0,k]$, we have
\[
\int_{G} \abs[\Big]{ \int_{G^{k}} \prod_{i=0}^{k} f_{i}T_{[1,i]}^{g} } \dif g
\leq
\min_{i} \norm{f_{i} N^{[0,k]}_{i}}_{\square^{k}} + C_{k}D^{-2^{-k}}.
\]
\end{proposition}

\begin{proof}
By induction on $k$.
For $k=1$ the box norm is just the absolute value of the integral, so writing
\[
\int_{G} \abs[\Big]{ \int_{G^{1}} \prod_{i=0}^{1} f_{i}T_{[1,i]}^{g} } \dif g
\leq
\int_{G} \abs[\Big]{ \int_{G^{1}} f_{0} \cdot f_{1}T_{1}^{g} - \int f_{0} \int f_{1} } \dif g
+
\abs[\Big]{\int f_{0}} \abs[\Big]{\int f_{1}}
\]
we can estimate the second term by the minimum of the box norms.
In the first term we apply Jensen's inequality and Lemma~\ref{lem:density-mean-ergodic} with the Hilbert space $L^{2}(G)$ and the unitary right $G$-action $(g,f) \mapsto fT_{1}^{g}$.
Since the invariant subspace of this action consists only of the constant functions, the projection onto this subspace amounts to integration over $G$.

Suppose now $k>1$ and the claim is known for $k-1$.
Applying $T^{g^{-1}}_{[1,k]}$ to the function in the inner integral and reversing the order of the indices $0,\dotsc,k$ we see that the bound with $f_{0}$ follows from the bound with $f_{k}$, so it suffices to establish bounds with $f_{1},\dotsc,f_{k}$.

Applying Lemma~\ref{lem:vdC} with $X=G$, $V=L^{2}(G^{k})$, $v=f_{0}$, and $u(g)=\prod_{i=1}^{k} f_{i}T_{[1,i]}^{g}$, we estimate the square of the left-hand side of the conclusion by
\[
\int_{h}\int_{g} \abs[\Big]{ \int_{G^{k}} \prod_{i=1}^{k} f_{i}T_{[1,i]}^{g} \cdot f_{i}T_{[1,i]}^{h} }
=
\int_{h}\int_{g} \abs[\Big]{ \int_{G^{k}} \prod_{i=1}^{k} (f_{i} \cdot f_{i}T_{[1,i]}^{h})T_{[2,i]}^{g} }.
\]
In the last step we have made the change of variables $(g,h)\mapsto (g,hg)$ on $G^{2}$ and used the fact that $T_{1}^{g}$ is a measure-preserving transformation.
Pulling one of the integrals out of the absolute value we obtain the estimate
\[
\int_{h}\int_{g} \int_{a_{1}}\abs[\Big]{ \int_{a_{2},\dotsc,a_{k}} \prod_{i=1}^{k} (f_{i} \cdot f_{i}T_{[1,i]}^{h})T_{[2,i]}^{g} (a_{1},\dotsc,a_{k}) }.
\]
Applying the inductive hypothesis for each fixed pair $(h,a_{1})$, for any $i\in [1,k]$ we obtain the estimate
\[
\int_{h} \int_{a_{1}} \norm{(f_{i} \cdot f_{i}T_{[1,i]}^{h})(a_{1},N^{[1,k]}_{i}\cdot)}_{\square^{k-1}} + C_{k-1} D^{-2^{-k+1}},
\]
where $N^{[1,k]}_{i}(x_{1},\dotsc,x_{i-1},x_{i+1},\dotsc,x_{k})$ is defined similarly to $N^{[0,k]}_{i}$.

The contribution of the second summand is admissible, so we only have to consider the first summand.
Raising it to the power $2^{k-1}$ and applying Jensen's inequality we obtain the bound
\[
\int_{h,a_{1}} \norm{(f_{i} \cdot f_{i}T_{[1,i]}^{h})(a_{1},N^{[1,k]}_{i} \cdot)}_{\square^{k-1}}^{2^{k-1}}.
\]
Expanding the definition of the box norm and observing that
\[
T_{[1,i]}^{h}(a_{1},N^{[1,k]}_{i} \vec x_{(i)})
=
N^{[0,k]}_{i}(ha_{1},a_{1}^{-1}x_{1},x_{2},\dotsc,x_{i-1},x_{i+1},\dotsc,x_{k})
\]
we can write the above expression in the form
\begin{multline*}
\int_{h,a_{1}} \int
\prod_{\vec\epsilon\in\Set{0,1}^{[1,k]\setminus \Set{i}}}
f_{i}N^{[0,k]}_{i}(a_{1}, a_{1}^{-1} x_{1,\epsilon_{1}}, x_{2,\epsilon_{2}},\dotsc,x_{k,\epsilon_{k}})\\
\cdot f_{i}N^{[0,k]}_{i}(ha_{1}, a_{1}^{-1} x_{1,\epsilon_{1}}, x_{2,\epsilon_{2}},\dotsc,x_{k,\epsilon_{k}})
\dif (x_{j,\epsilon})_{j\in [1,k]\setminus \Set{i}, \epsilon\in\Set{0,1}}.
\end{multline*}
With the change of variables $(a_{1},x_{1,0},x_{1,1})\mapsto (a_{1},a_{1}x_{1,0},a_{1}x_{1,1})$ this becomes
\[
\int_{h,a_{1}} \int
\prod_{\epsilon} f_{i}N^{[0,k]}_{i}(a_{1}, \vec x^{\epsilon}_{(i)})
\cdot
\prod_{\epsilon} f_{i}N^{[0,k]}_{i}(ha_{1}, \vec x^{\epsilon}_{(i)})
\dif (x_{j,\epsilon})_{j\in [1,k]\setminus \Set{i}, \epsilon\in\Set{0,1}}.
\]
We interpret the integral in all variables but $h$ as an inner product in $L^{2}(G^{2k-1})$ and the appearance of $h$ in the first argument of the second product as a right unitary action of $G$ on this space.
Applying Lemma~\ref{lem:density-mean-ergodic}, we obtain an admissible error term and the bound
\[
\int_{a_{1}} \int
P\Big(\prod_{\epsilon} f_{i}N^{[0,k]}_{i}\Big)^{2}(a_{1}, \vec x^{\epsilon}_{(i)})
\dif (x_{j,\epsilon})_{j\in [1,k]\setminus \Set{i}, \epsilon\in\Set{0,1}},
\]
where $P$ denotes the projection onto the invariant subspace.
But this projection is nothing else but integration in the variable $a_{1}$, so this can be written as
\[
\int
\Big(\int_{a_{1}} \prod_{\epsilon} f_{i}N^{[0,k]}_{i}(a_{1}, \vec x^{\epsilon}_{(i)}) \Big)^{2}
\dif (x_{j,\epsilon})_{j\in [1,k]\setminus \Set{i}, \epsilon\in\Set{0,1}}.
\]
Relabeling $a_{1}=x_{0,0}$ in the first factor of the square and $a_{1}=x_{0,1}$ in the second factor, we see that this coincides with
\[
\norm{ f_{i}N^{[0,k]}_{i} }_{\square^{k}}^{2^{k}},
\]
and the conclusion follows.
\end{proof}

\begin{proof}[Proof of Theorem~\ref{thm:multiple-weak-mixing}]
By induction on $k$.
The base case $k=0$ is very easy.

Let now $k\geq 1$ and suppose that the result holds for $k-1$.
Let $\epsilon>0$ be arbitrary and apply the weak regularity lemma to the function $f_{k}N^{[0,k]}_{k}$, so that
\[
f_{k}N^{[0,k]}_{k} = \sum_{l\in L}\prod_{i=0}^{k-1}F_{i,l} + F_{u},
\]
where all functions on the right-hand side are uniformly bounded,
the functions $F_{i,l}$ do not depend on the $i$-th and the $k$-th coordinates,
the index set $L$ has size $O_{\epsilon}(1)$,
and $\norm{F_{u}}_{\square^{k}} \leq\epsilon$.
Split the multicorrelation sequence accordingly as
\[
c_{g} = \sum_{l\in L} c_{l,g} + c_{u,g}.
\]
The inverse of the change of variables $N^{[0,k]}_{i}$ is given by
\[
(N^{[0,k]}_{i})^{-1}(\vec a)
=
(a_{1},a_{1}^{-1}a_{2},\dotsc,a_{i-1}^{-1}a_{i},a_{i+1}^{-1}a_{i+2},\dotsc,a_{k-1}^{-1}a_{k},a_{k}^{-1}),
\]
and it can be verified that we have
\[
((N^{[0,k]}_{k})^{-1} T_{[1,k]}^{g})_{(i)} = ((N^{[0,k]}_{k})^{-1} T_{[1,i]}^{g})_{(i)},
\]
thus the actions $T_{[1,k]}$ and $T_{[1,i]}$ coincide on the functions $F_{i,l}$.
This is a common theme in the hypergraph regularity approach to multiple ergodic averages in the work of Austin (although it takes much less effort to exploit this phenomenon in our compact group setting than in the setting of infinite amenable groups).
Since the maps $f\mapsto fT_{[1,i]}^{g}$ are algebra homomorphisms, it follows that
\[
c_{l,g} = \int_{G^{k}} \prod_{i=0}^{k-1} (f_{i} \cdot F_{i,l} (N^{[0,k]}_{k})^{-1}) T_{[1,i]}^{g}.
\]
This is an average (in the last coordinate of $G^{k}$) of multicorrelation sequences of length $k-1$, so its total variation is bounded by $o_{D\to\infty;k-1}(1)$ by the inductive hypothesis.
On the other hand, we have
\[
\int_{g} \abs{c_{u,g}} \leq \epsilon + o_{D\to\infty;k}(1)
\]
by Proposition~\ref{prop:box-norm-controls-multicorrelations}.
This shows that the total variation of the multicorrelation sequence $c_{g}$ can be estimated by
\[
\card{L} o_{D\to\infty;k-1}(1) + \epsilon + o_{D\to\infty;k}(1)
=
\epsilon + o_{D\to\infty;k,\epsilon}(1).
\]
Since $\epsilon>0$ was arbitrary this concludes the proof.
\end{proof}

\printbibliography
\end{document}
